\documentclass{article}
\usepackage[utf8]{inputenc}
\usepackage{amsmath}
\usepackage{amsthm}
\usepackage{amsfonts}
\usepackage{amssymb}
\usepackage[T1]{fontenc}
\usepackage[all]{xy}
\usepackage{bbm}
\usepackage{enumitem}
\usepackage{color}
\usepackage[hidelinks]{hyperref}
\usepackage{graphicx}
\usepackage{transparent}
\usepackage{graphics}
\usepackage{xcolor}
\usepackage{color}
\usepackage{geometry}
\usepackage{marginnote}
\usepackage{extarrows}
\usepackage{pgf,tikz}
\usepackage{mathrsfs}
\usepackage{mathtools}
\usepackage{float}
\usepackage[citestyle=numeric,bibstyle=numeric]{biblatex}
\usetikzlibrary{arrows,decorations.pathmorphing,backgrounds,positioning,fit,petri,cd}

\theoremstyle{plain} 
\newtheorem{theorem}{Theorem}[section]
\newtheorem{proposition}[theorem]{Proposition}
\newtheorem{lemma}[theorem]{Lemma}

\DeclareMathOperator{\Hom}{Hom}
\DeclareMathOperator{\Ext}{Ext}

\theoremstyle{remark}
\newtheorem{remark}[theorem]{Remark}

\theoremstyle{definition}
\newtheorem{definition}[theorem]{Definition}

\newtheorem{construction}[theorem]{Construction}

\title{The Grothendieck Groups Of Discrete Cluster Categories of Dynkin Type \texorpdfstring{$A_{\infty}$}{Ainf}}
\author{Dave Murphy}
\date{}

\begin{document}

\maketitle

\begin{abstract}
In this work we compute the triangulated Grothendieck groups for each of the family of discrete cluster categories of Dynkin type $A_{\infty}$ as introduced by Holm-J\o rgensen.
Subsequently, we also compute the Grothendieck group of a completion of these discrete cluster categories in the sense of Paquette-Yildirim.

\end{abstract}

\section{Introduction}

Cluster categories were introduced in \cite{Buan2007} as an orbit category of the bounded derived category of the module category of a finite-dimensional hereditary algebra $\mathcal{H}$ over a field $k$.
These provide a setting in which to study the tilting theory of $\mathcal{H}$, and are a model for the combinatorics of an associated Fomin-Zelevinsky cluster algebra.

Discrete cluster categories of Dynkin type $A$ are a family of cluster categories introduced by Igusa-Todorov around 2013 \cite{Igusa2013}.
A discrete cluster category of Dynkin type $A_{\infty}$, $\mathcal{C}(\mathscr{Z})$, is a triangulated category associated to a fixed discrete subset $ \mathscr{Z} \subset S^1$ of the unit circle, which has a finite number of accumulation points.
The indecomposable objects of $\mathcal{C}(\mathscr{Z})$ are in a bijection with the arcs between marked points on $S^1$, with $\mathrm{Ext}$-spaces non-vanishing when two arcs cross.
We denote by $\mathcal{C}_n$ the category $\mathcal{C}(\mathscr{Z})$ when $\mathscr{Z}$ has $n$ two-sided accumulation points.

Discrete cluster categories are $\mathrm{Hom}$-finite, Krull-Schmidt, $k$-linear, 2-Calabi-Yau triangulated categories.
This means that they are very well behaved categories, which when considered with their association to the arcs of the discrete subset $\mathscr{Z} \subset S^1$, makes for a family of cluster categories which are easy to work with and give an insight into the behaviour of cluster categories as a whole.
This makes them a powerful tool in the development and understanding of cluster categories.
Recent work on discrete cluster categories includes \cite{Holm2009}, \cite{Gratz2017}, \cite{GZ2021}, \cite{Igusa2013} and \cite{Baur2016}.
Something we also wish to consider is the completion of a discrete cluster category, as studied in \cite{Paquette2020} and \cite{Fisher2014}, which is the category found by formally adding arcs going into the accumulation points.

The Grothendieck group of a triangulated category, $\mathcal{T}$, is defined to be the free abelian group of isomorphism classes of $\mathcal{T}$, modulo the \textit{Euler relations}: $[B]=[A]+[C]$ for each triangle $A \rightarrow B \rightarrow C \rightarrow \Sigma A$ in $\mathcal{T}$.
The Grothendieck group, denoted $K_0^{\operatorname{add}}(\mathcal{T})$, has the universal property that given a function from the isomorphism classes of objects in $\mathcal{T}$ to some abelian group $G$ such that the Euler relations are preserved, then said function factors through a unique homomorphism $K_0^{\operatorname{add}}(\mathcal{T}) \rightarrow G$.

There has been much work done in recent years on the Grothendieck group of cluster categories, and on triangulated categories in general: such as in \cite{Palu2008}, \cite{Naisse2019} and \cite{Fedele2018}.
One can identify the subgroups of the Grothendieck group of an essentially small triangulated category with the dense, strictly full triangulated subcategories of said category, as shown in \cite{Thomason1997}.
This reduces the problem of classifying the dense, strictly full triangulated subcategories to the much more straightforward task of computing the Grothendieck group and its respective subgroups.

In \cite{Barot07}, a computation of the Grothendieck group for any cluster category of a Dynkin quiver was presented. 
Of particular note to this work, they found that the Grothendieck group of a cluster category of $A_n$ alternates between $\mathbb{Z}$ for odd $n$ and trivial for even $n$.
While one may naively believe that a similar periodicity might appear in the Grothendieck groups of $\mathcal{C}_n$, we show that this is not the case; in fact we show

\begin{theorem}[Theorem \ref{Thm:TriGroGroupCn}]
Let $\mathscr{Z}$ be a subset of $S^1$ such that there are exactly $n$ two-sided accumulation points. Then the Grothendieck group of $\mathcal{C}_n$ is
\[
K_0(\mathcal{C}_n) \cong \mathbb{Z}^n
\]
\end{theorem}

Another result which is made use of in this work is due to Palu \cite{Palu2008}, in which they show that the triangulated Grothendieck group of a triangulated category is isomorphic to the quotient of the split Grothendieck group of a cluster tilting subcategory by a set of specified relations. 
We make heavy use of this result in the proof of Theorem \ref{Thm:TriGroGroupCn}.

Further, we build upon this result by using it to compute the Grothendieck group of a completion of the discrete cluster category, in the sense of Paquette-Yildirim
\cite{Paquette2020}, with $n$ two-sided accumulation points, denoted $\overline{\mathcal{C}}_n$.

\begin{theorem}[Theorem \ref{Thm2}]
Let $\overline{\mathcal{C}}_n$ be the completion of an Igusa-Todorov cluster category with $n$ two-sided accumulation points.
Then the triangulated Grothendieck group of $\overline{\mathcal{C}}_n$ is:
\[
K_0(\overline{\mathcal{C}}_n) \cong \mathbb{Z}^n \oplus (\mathbb{Z}/2\mathbb{Z})^{n-1}
\]
\end{theorem}

This work is organised as follows; Section \ref{Sec:2} will be a brief discussion on the construction of the subset of $S^1$ associated to the discrete cluster categories of Dynkin type $A_{\infty}$.
For this, we use the definitions found in \cite{Gratz2017} and subsequently follow the constructions as outlined by them.

Section \ref{Sec:3} is where we consider the cluster tilting subcategories of $\mathcal{C}_n$.
Using an appropriate cluster tilting subcategory, as well as its mutations, gives us a way of computing the Grothendieck group of $\mathcal{C}_n$ via a theorem thanks to Palu \cite{Palu2008}.
A complete description of the cluster tilting subcategories is again provided in \cite{Gratz2017}.

Our work begins in Section \ref{Sec:4}, where we compute the Grothendieck group of $\mathcal{C}_n$ for all $n \geq 1$.
We build upon this in Section \ref{Sec:5}, where we look at the completion of the discrete cluster category of Dynkin type $A_{\infty}$, denoted $\overline{\mathcal{C}}_n$, as well as computing its Grothendieck group.
We first go over a brief construction of $\overline{\mathcal{C}}_n$ in the sense of \cite{Paquette2020}, followed by some lemmas necessary to the proof of the Grothendieck groups of $\overline{\mathcal{C}}_n$.
We finish this work with the proof of Theorem \ref{Thm2}.

\subsection*{Acknowledgements}

The author would like to express many thanks to their supervisors, Sira Gratz and Greg Stevenson, for their help and guidance throughout the process of this work, as well as introducing the author to the fascinating world of cluster categories.

The author would also like to thank Yann Palu, not only for the original work that this paper is based upon, but also their grace and understanding when an issue in their statement of Theorem \ref{Thm:Palu} was pointed out.
Thanks must also go to Matthew Pressland for conversations regarding the aforementioned issue, and in helping the author understand the its true nature.

\section{Discrete Cluster Categories of Dynkin Type \texorpdfstring{$A_{\infty}$}{Ainf}}\label{Sec:2}

We start with a brief introduction to the discrete cluster category $\mathcal{C}_n$ of type $A_{\infty}$, as introduced in \cite{Igusa2013}, as well as some other elementary definitions.
We will follow the description of $\mathcal{C}_n$ found in \cite{Gratz2017}, which also provides for many of the other definitions.
This is so one can see the cluster tilting subcategories as maximal collections of non-crossing arcs \cite{Gratz2017} (also see Section \ref{Sec:3}).

One can see the discrete cluster categories via the combinatorics of an $\infty$-gon, with indecomposable objects being arcs non-isotopic to those arcs between consecutive marked points, and $\mathrm{Ext}^1$ spaces between two arcs being non-trivial if and only if the arcs cross.

\subsection{The Marked Points of \texorpdfstring{$S^1$}{S1}}

We formalise this intuitive approach with the following definitions:

\begin{definition}\cite{Igusa2013}
A subset $\mathscr{Z}$ of the circle $S^1$ is called \textit{admissible} if it satisfies the following conditions.
\begin{enumerate}
    \item $\mathscr{Z}$ has infinitely many elements,
    \item $\mathscr{Z} \subset S^1$ is a discrete subset,
    \item $\mathscr{Z}$ satisfies the \textit{two-sided limit condition}, i.e. each $x \in S^1$ which is the limit of a sequence in $\mathscr{Z}$ is a limit of both an increasing and decreasing sequence from $\mathscr{Z}$ with respect to the cyclic order.
\end{enumerate}
\end{definition}
We shall fix $\mathscr{Z} \subset S^1$ throughout, which can be thought of as the vertices of an $\infty$-gon.
The topological closure of $\mathscr{Z}$ is denoted by $\overline{\mathscr{Z}}$.
We denote by $[x,z] \subset \overline{\mathscr{Z}}$ the closed set of marked points from $x$ to $z$ in an anti-clockwise direction.
The half-open sets $(x,z], [x,z)$ and the open set $(x,z)$ are defined analogously.
Given three marked points, $x,y,z \in \mathscr{Z}$, we say that $x < y < z$ if $y \in (x,z)$.

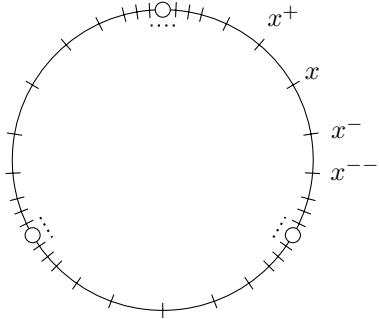
\begin{figure}[ht]
\centering
\begin{tikzpicture}
\draw (2,2) circle (2cm)
(2.36,4.07) -- (2.33,3.87)
(2.49,3.84) -- (2.54,4.03)
(2.17,3.9) -- (2.18,4.1)
(2.89,3.90) -- (2.80,3.72)
(3.22,3.46) -- (3.35,3.61)
(3.82,3.05) -- (3.65,2.95)
(3.87,2.33) -- (4.07,2.36)
(4.09,1.82) -- (3.89,1.83)
(3.84,1.51) -- (4.03,1.46)
(3.97,1.28) -- (3.79,1.35)
(3.72,1.20) -- (3.90,1.11)
(3.72,0.80) -- (3.56,0.91)
(3.46,0.78) -- (3.61,0.65)
(3.48,0.52) -- (3.34,0.66)
(3.09,0.44) -- (3.20,0.28)
(2.72,0.03) -- (2.65,0.21)
(2,0.1) -- (2,-0.1)
(1.28,0.03) -- (1.35,0.21)
(0.91,0.44) -- (0.80,0.28)
(0.52,0.52) -- (0.66,0.66)
(0.39,0.65) -- (0.55,0.78)
(0.44,0.91) -- (0.28,0.80)
(0.10,1.11) -- (0.28,1.20)
(0.21,1.35) -- (0.03,1.28)
(-0.03,1.46) -- (0.16,1.51)
(0.11,1.83) -- (-0.09,1.82)
(-0.07,2.36) -- (0.13,2.33)
(0.35,2.95) -- (0.18,3.05)
(0.65,3.61) -- (0.78,3.46)
(1.20,3.72) -- (1.11,3.90)
(1.46,4.03) -- (1.51,3.84)
(1.67,3.87) -- (1.64,4.07)
(1.82,4.09) -- (1.83,3.89);
\node at (3.99,3.15) {$x$};
\node at (3.61,3.92) {$x^+$};
\node at (4.46,2.43) {$x^-$};
\node at (4.56,1.88) {$x^{--}$};
\draw[dotted, thick] (2.16,3.79) -- (1.84,3.79);
\draw[dotted, thick] (3.47,0.97) -- (3.63,1.24);
\draw[dotted, thick] (0.37,1.24) -- (0.53,0.97);
\draw[fill=white] (0.27,1) circle (0.1cm);
\draw[fill=white] (3.73,1) circle (0.1cm);
\draw[fill=white] (2,4) circle (0.1cm);
\end{tikzpicture}
\caption{An example of an admissible subset $\mathscr{Z} \subset S^1$, thought of as the vertices of an $\infty$-gon. The small circles denote the points to which the points in $\mathscr{Z}$ converge, we call these accumulation points. Since $\mathscr{Z}$ is discrete, these accumulation points are not part of $\mathscr{Z}$.}
\label{fig:Example}
\end{figure}

\begin{definition}\cite[Definition 0.2]{Gratz2017}
A \textit{arc of} $\mathscr{Z}$ is a subset $X=\{x_0,x_1\} \subset \mathscr{Z}$ where $x_1 \notin \{x_0^-, x_0, x_0^+\}$, where $x^+$ and $x^-$ are the successor and predecessor respectively to $x \in \mathscr{Z}$.
If $Y= \{y_0,y_1\}$ is another arc, then $X$ and $Y$ \textit{cross} if $x_0<y_0<x_1<y_1 <x_0$ or $x_0<y_1<x_1<y_0<x_0$.
\end{definition}

We can think of these arcs as lines on Figure \ref{fig:Example} connecting any two points in $\mathscr{Z}$, and arcs cross exactly when these lines intersect each other.
We consider the arcs up to isotopy. 
Note that the definition precludes arcs starting and/or ending at the same point from being considered to cross.

Finally, when we are discussing these categories we need a way of describing the accumulation points, as well as a way of describing the behaviour of a set of arcs in relation to these accumulation points.
Given the topological closure $\overline{\mathscr{Z}}$ of $\mathscr{Z}$, we denote 
\[
L(\mathscr{Z}) = \overline{\mathscr{Z}}\backslash \mathscr{Z}
\]
the discrete set of proper accumulation points of $\mathscr{Z}$.

For convenience, we may impose a cyclic ordering upon $L(\mathscr{Z})$, labelling the accumulation points $a_1, \ldots, a_n$.
Note that due to the discrete nature of $\mathscr{Z}$, this means that it is disjoint from $L(\mathscr{Z})$.

Next we give some names to specific collections of arcs that we may wish to use.

\begin{definition}\cite[Definition 1.4]{Gratz2017}
Let $\{x_i\}_{i \in \mathbb{Z}_{i \geq 0}}$ be a convergent sequence in $\mathscr{Z}$.
If $\{x_i\}_{i \in \mathbb{Z}_{i \geq 0}}$ converges to $p \in \overline{\mathscr{Z}}$, then we write $x_i \rightarrow p$.
\begin{itemize}
    \item We say that $x_i \rightarrow p$ \textit{from below} if there is a $\mu \in S^1 \backslash \{p\}$ such that $x_i \in [\mu, p]$ from some step.
    \item We say that $x_i \rightarrow p$ \textit{from above} if there is a $\nu \in S^1 \backslash \{p\}$ such that $x_i \in [p, \nu]$ from some step.
\end{itemize}
\end{definition}

\begin{definition}\cite[Definition 0.4]{Gratz2017}
Let $\mathscr{X}$ be a set of arcs of $\mathscr{Z}$.
\begin{itemize}
    \item Given $a \in L(\mathscr{Z})$, we say that $\mathscr{X}$ \textit{has a leapfrog converging to} $a \in L(\mathscr{Z})$ if there is a sequence $\{x_i,y_i\}_{i \in \mathbb{Z}_{i \geq 0}}$ of arcs from $\mathscr{X}$ with $x_i \rightarrow a$ from below and $y_i \rightarrow a$ from above.
    \item Given $a \in L(\mathscr{Z})$, $z \in \mathscr{Z}$, we say that $\mathscr{X}$ \textit{has a right fountain at} $z$ \textit{converging to} $a$ if there is a sequence $\{z,x_i\}_{i \in \mathbb{Z}_{i \geq 0}}$ from $\mathscr{X}$ with $x_i \rightarrow a$ from below. We say that $\mathscr{X}$ \textit{has a left fountain at} $z$ \textit{converging to} $a$ if there is a sequence $\{z,y_i\}_{i \in \mathbb{Z}_{i \geq 0}}$ from $\mathscr{X}$ with $y_i \rightarrow a$ from above.
    \item We say that $\mathscr{X}$ \textit{has a fountain at} $z$ \textit{converging to} $a$ if it has a right fountain and a left fountain at $z$ converging to $a$.
\end{itemize}
\end{definition}

\subsection{The Category \texorpdfstring{$\mathcal{C}_n$}{Cn}}

Given $\mathscr{Z}$ and a field $k$, Igusa and Todorov \cite{Igusa2013} constructed a cluster category $\mathcal{C}(\mathscr{Z})$ of Dynkin type $A_{\infty}$ for a fixed subset $\mathscr{Z} \subset S^1$.
This is constructed as a $k$-linear, $\mathrm{Hom}$-finite, Krull-Schmidt, 2-Calabi-Yau triangulated category, where indecomposables are in bijection to the arcs, with non-vanishing $\mathrm{Ext}^1$ groups and crossing arcs in correspondence.
More specifically, we have 
\[
\mathrm{Ext}^1(X,Y) =
\begin{cases*}
k & \text{if} X,Y \text{cross}\\
0 & \text{else}
\end{cases*}
\]
Whereby an abuse of notation, here $X$ and $Y$ denotes both the object in $\mathcal{C}(\mathscr{Z})$ and the corresponding arc in $\mathscr{Z}$.

Throughout we shall instead refer to $\mathcal{C}_n$, meaning the category $\mathcal{C}(\mathscr{Z})$ given a fixed admissible subset $\mathscr{Z}$ with $n$ two-sided accumulation points.

The suspension functor, $\Sigma$, of $\mathcal{C}_n$, acts upon an object $X \in \mathcal{C}_n$ by rotating the corresponding arc around by one marked point in a clockwise direction, i.e. if $X$ corresponds to the arc $\{x_0,x_1\}$, then $\Sigma X$ corresponds to the arc $\{x_0^-,x_1^-\}$.

Crossing arcs in $\mathscr{Z}$ induce triangles in $\mathcal{C}_n$.
For a pair of crossing arcs $M$ and $N$, we can form two quadrilaterals that induce two different triangles in $\mathcal{C}_n$, these quadrilaterals have the form:

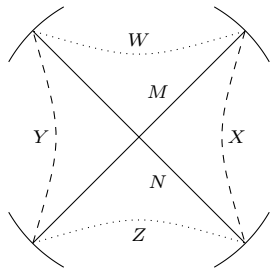
\[
\begin{tikzpicture}
   \draw [domain=30:60] plot ({2*cos(\x)}, {2*sin(\x)});
   \draw [domain=210:240] plot ({2*cos(\x)}, {2*sin(\x)});
    \draw [domain=300:330] plot ({2*cos(\x)}, {2*sin(\x)});
   \draw [domain=120:150] plot ({2*cos(\x)}, {2*sin(\x)});
   \draw[thin,dashed] (1.41,1.41)..controls(1,0)..(1.41,-1.41);
   \draw[thin,dotted] (1.41,1.41)..controls(0,1)..(-1.41,1.41);
    \draw[thin,dotted] (-1.41,-1.41)..controls(0,-1)..(1.41,-1.41);
   \draw[thin,dashed] (-1.41,-1.41)..controls(-1,0)..(-1.41,1.41);
   \draw[thin] (1.41,1.41) -- (-1.41,-1.41);
   \draw[thin] (1.41,-1.41) -- (-1.41,1.41);
   \node at (1.3,0) {\scriptsize $X$};
   \node at (0,1.3) {\scriptsize $W$};
    \node at (-1.3,0) {\scriptsize $Y$};
   \node at (0,-1.3) {\scriptsize $Z$};
   \node at (0.25,0.6) {\scriptsize $M$};
   \node at (0.25,-0.6) {\scriptsize $N$};
\end{tikzpicture}
\]
where we have a quadrilateral with the four arcs $M,N,X,Y$ as sides, and another with the arcs $N,M,W,Z$ as sides.
The two triangles induced by the crossing arcs are;

\begin{align*}
    M \rightarrow X \oplus Y \rightarrow N \rightarrow \Sigma M\\
    N \rightarrow W \oplus Z \rightarrow M \rightarrow \Sigma N\\
\end{align*}

Finally, we would sometimes like to consider some specific subcategories of $\mathcal{C}_n$, ones with an equivalence to $\mathcal{C}_1$.
For this purpose, we have the following definition.

\begin{definition}
Let $\mathscr{Z}$ be an admissible subset of $S^1$, then we say a \textit{segment} of $S^1$ is the open set $(a_i,a_{i+1}) \subset S^1$, with $a_i,a_{i+1} \in L(\mathscr{Z})$.
If $i=n$, then the segment is defined as the open set $(a_n,a_1)$.
Note that if there is only one limit point, then we say the only segment of $S^1$ is $S^1 \backslash L(\mathscr{Z})$.
\end{definition}

One may see that there is an injection $\tau: \mathcal{C}_1 \rightarrow \mathcal{C}_n$, such that $\tau(\mathcal{C}_1)$ is equivalent to a segment $(a_i, a_{i+1})$ for some $i \in \{1,\ldots,n\}$.

\section{Cluster Tilting Subcategories}\label{Sec:3}

For the rest of this work, we shall assume all of our subcategories to be full.

Let $\mathcal{T}$ be a triangulated category with suspension functor $\Sigma$, and let $\mathcal{S}$ be a subcategory of $\mathcal{T}$.
Let $f \in \mathrm{Hom}_{\mathcal{T}}(X,Y)$, then, we call $f$ a \textit{right} $\mathcal{S}$-\textit{approximation} of $Y \in \mathcal{T}$ if $X \in \mathcal{S}$ and 
\[
\mathrm{Hom}(-,X) \xrightarrow{f \cdot} \mathrm{Hom}(-,Y) \rightarrow 0
\]
is exact as functors on $\mathcal{S}$.
We call $\mathcal{S}$ a \textit{contravariantly finite subcategory} of $\mathcal{T}$ if any $Y \in \mathcal{T}$ has a right $\mathcal{S}$-approximation.
A \textit{left} $\mathcal{S}$-\textit{approximation} and a \textit{covariantly finite subcategory} are dually defined.
We say that a contravariantly and covariantly finite subcategory is \textit{functorially finite}.

Contravariantly finite subcategories were first introduced in \cite{Auslander1980}, with respect to finitely generated module categories of artin algebras.

Now, let $\mathcal{S}$ be a functorially finite subcategory of $\mathcal{T}$, then we say $\mathcal{S}$ is a \textit{cluster tilting subcategory} of $\mathcal{T}$ if
\[
\mathcal{S}=(\Sigma^{-1}\mathcal{S})^{\perp}={}^{\perp}(\Sigma \mathcal{S})
\]
where
\[
X^{\perp}=\{y \in \mathcal{T} \; \vline \; \mathrm{Hom}(x,y)=0\; \text{for all}\; x \in X\} 
\]
and 
\[
{}^{\perp}X =\{ y \in \mathcal{T}\; \vline \; \mathrm{Hom}(y,x)=0\; \text{for all}\; x \in X\}
\]
for a subcategory $X$ of $\mathcal{T}$.

\subsection{Cluster Tilting Subcategories of \texorpdfstring{$\mathcal{C}_n$}{C(Z)}}

When we are working in a 2-Calabi-Yau category, it is clear that $(\Sigma^{-1}\mathcal{S})^{\perp}={}^{\perp}(\Sigma \mathcal{S})$, and so we only need to show that $\mathcal{S}={}^{\perp}(\Sigma \mathcal{S})$ for $\mathcal{S}$ to be a cluster tilting subcategory.
The cluster tilting subcategories of $\mathcal{C}_n$ were classified by Gratz-Holm-J\o rgensen in \cite{Gratz2017} for any $n \geq 1$.

We denote by $\operatorname{add}(\mathscr{X})$ the additive category for some collection of objects $\mathscr{X} \subset \operatorname{Ob}(\mathcal{C}_n)$.

\begin{theorem}\cite[Theorem 5.7]{Gratz2017} \label{Thm:GHJ}
Let $\mathscr{X}$ be a set of arcs of $\mathscr{Z}$, and let $\mathscr{X}$ be the set of objects in $\mathcal{C}_n$ corresponding to the set $\mathscr{X}$.
Then $\operatorname{add} (\mathscr{X})$ is a cluster tilting subcategory of $\mathcal{C}_n$ if and only if $\mathscr{X}$ is a maximal set of pairwise non-crossing arcs, such that for each $a \in L(\mathscr{Z})$, the set $\mathscr{X}$ has a fountain or leapfrog converging to $a$.
\end{theorem}

In Section \ref{Sec:4}, we will need to chose a cluster tilting subcategory of $\mathcal{C}_n$ to work with, for each $n \geq 1$.
For some technical reasons, it makes more sense for us to consider a cluster tilting subcategory with a leapfrog converging to each accumulation point rather than a fountain.
With this in mind, from now on we shall fix a cluster tilting subcategory of $\mathcal{C}_n$.

\begin{construction}
Given we have $n$ accumulation points in $\mathscr{M}$ then we have $L(\mathscr{M})=\{a_1,\ldots,a_n\}$.
Let $\{z_1, \ldots, z_n\} \subset \mathscr{Z}$ be a set of marked points such that $z_i \in (a_i,a_{i+1})$ for each segment of $S^1$.
By taking the arcs $\{z_i, z_{i+1}\}$ for all $i = 1,\ldots, n$ we inscribe an $n$-gon inside the circle; the arc $\{z_i,z_{i+1}\}$ shall correspond to the object $Z_i \in \mathcal{C}_n$ for all $i=1, \ldots, n$.
For all $a_i \in L(\mathscr{Z})$, let $\{z_i,z_{i+1}\}$ be the first arc in a leapfrog converging to $a_{i+1}$, such that if the arc $\{x,y\}$ is in the leapfrog, then so are either the arcs $\{x^+,y\}$ and $\{x,y^+\}$, or the arcs $\{x^-,y\}$ and $\{x,y^-\}$.
We shall label the leapfrog converging to $a_i$ by $\mathscr{L}_i$.
Without loss of generality, we may let $\{z_i, z_{i+1}^-\}$ be in $\mathscr{L}_{i+1}$, which has a corresponding object labelled $Y_i$.
Notice that this implies that for every arc $M \in \mathscr{L}_{i+1}$ that is not $Z_i$, there are exactly two other objects, $M^-,M^+ \in \mathscr{L}_{i+1}$ that share an endpoint with $M$.

For $n \geq 4$, we give the inscribed $n$-gon a fan triangulation; that is, we add the arcs $\{z_1,z_i\}$ for all $i=2,\ldots, n$.
We shall label the object correspond to the arc $\{z_1,z_i\}$ as $X_i$, for all $i=2,\ldots,n$.
Notice that this means that we have $X_2 \cong Z_1$ and $X_n \cong Z_n$.
This collection of arcs will look like Figure \ref{fig:triangulation}, along with a leapfrog $\mathscr{L}_i$ converging to each accumulation point.
\end{construction}

\begin{figure}[h!]
\centering
\begin{tikzpicture}
\draw ({4*sin(350)},{4*cos(350)}) arc (100:440:4);
\draw[dotted] ({4*sin(10)},{4*cos(10)}) arc (80:100:4);
\draw[fill=white] ({4*sin(72)},{4*cos(72)}) circle (0.1cm);
\draw[fill=white] ({4*sin(144)},{4*cos(144)}) circle (0.1cm);
\draw[fill=white] ({4*sin(216)},{4*cos(216)}) circle (0.1cm);
\draw[fill=white] ({4*sin(288)},{4*cos(288)}) circle (0.1cm);
\draw ({4*sin(180)},{4*cos(180)}) .. controls(2,-2).. ({4*sin(98)},{4*cos(98)});
\draw ({4*sin(180)},{4*cos(180)}) -- ({4*sin(36)},{4*cos(36)});
\draw ({4*sin(180)},{4*cos(180)}) .. controls(-2,-2).. ({4*sin(252)},{4*cos(252)});
\draw ({4*sin(180)},{4*cos(180)}) -- ({4*sin(324)},{4*cos(324)});
\draw ({4*sin(180)},{4*cos(180)}) .. controls(2.2,-2.2).. ({4*sin(108)},{4*cos(108)});
\draw ({4*sin(190)},{4*cos(190)}) .. controls(-2.2,-2.2).. ({4*sin(252)},{4*cos(252)});
\draw ({4*sin(98)},{4*cos(98)}) .. controls(2.5,1.2).. ({4*sin(36)},{4*cos(36)});
\draw ({4*sin(98)},{4*cos(98)}) .. controls(2.8,1.4).. ({4*sin(46)},{4*cos(46)});
\draw ({4*sin(324)},{4*cos(324)}) .. controls(-2.2,1.4).. ({4*sin(252)},{4*cos(252)});
\draw ({4*sin(324)},{4*cos(324)}) .. controls(-2.5,1.6).. ({4*sin(262)},{4*cos(262)});
\draw[dashed] ({3.4*sin(288)},{3.4*cos(288)}) -- ({3.75*sin(288)},{3.75*cos(288)});
\draw[dashed] ({3.4*sin(216)},{3.4*cos(216)}) -- ({3.75*sin(216)},{3.75*cos(216)});
\draw[dashed] ({3.4*sin(144)},{3.4*cos(144)}) -- ({3.75*sin(144)},{3.75*cos(144)});
\draw[dashed] ({3.4*sin(72)},{3.4*cos(72)}) -- ({3.75*sin(72)},{3.75*cos(72)});
\node at (2,-1.7) {\footnotesize $Z_1$};
\node at (2.7,-2.3) {\footnotesize $Y_1$};
\node at (1,0) {\footnotesize $X_3$};
\node at (-0.8,0.1) {\footnotesize $X_{n-1}$};
\node at (-1.8,-2) {\footnotesize $Z_n$};
\node at (2.7,0.7) {\footnotesize $Z_2$};
\node at (3.1,1.5) {\footnotesize $Y_2$};
\node at (-2.4,0.2) {\footnotesize $Z_{n-1}$};
\node at (-3,1.6) {\footnotesize $Y_{n-1}$};
\node at (-2.4,-2.4) {\footnotesize $Y_n$};
\end{tikzpicture}
\caption{The collection of arcs corresponding to the cluster tilting subcategory we wish to consider for $n$ two-sided accumulation points, where leapfrogs are represented by the dashed lines. By Theorem \ref{Thm:GHJ}, this collection of leapfrogs together with the fan triangulation of the inscribed $n$-gon corresponds a cluster tilting subcategory of $\mathcal{C}_n$.}
\label{fig:triangulation}
\end{figure}
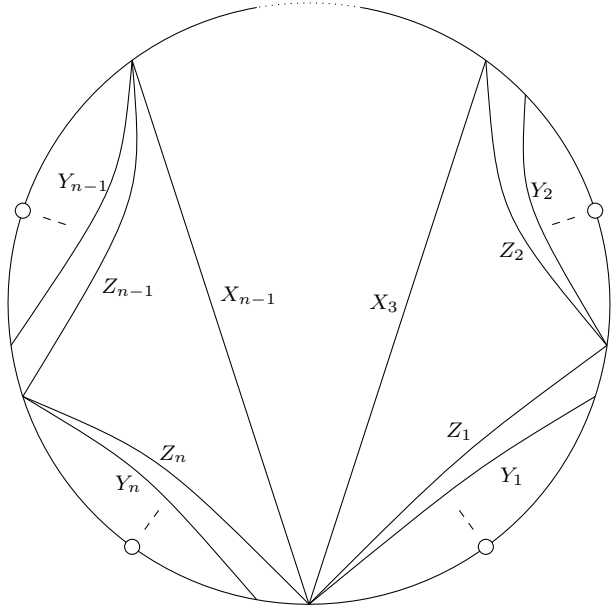

\subsection{Exchange Pairs}

Exchange pairs in a cluster category were introduced in \cite{BMRRT} as an extension of the notion of the completion of an almost complete basic tilting module over a hereditary algebra, to the theory of cluster categories and almost cluster tilting objects.
They also provide us with an analogue of the idea of mutation of clusters in a cluster algebra, and are hence useful to understand when looking at cluster categories generally.
More specifically, we wish to understand them as they play a crucial role in a result in \cite{Palu2008} which we will use in the proof of Theorem \ref{Thm:TriGroGroupCn}.

\begin{definition}
Let $\mathcal{T}$ be a cluster tilting subcategory of $\mathcal{C}_n$ with an infinite collection of indecomposable objects $\{T_i\}_{i \in I}$.
Let $\mathcal{T}' \subset \mathcal{T}$ be the full additive subcategory containing all but one $\{T_i\}_{i \in I}$, i.e. there exists exactly one indecomposable object $T \in \mathcal{T}$ such that $T \notin \mathcal{T}'$.
Then we call $\mathcal{T}'$ an \textit{almost cluster tilting subcategory}, and $T$ a \textit{complement} of $\mathcal{T}'$.
\end{definition}

By \cite{Holm2012} there are exactly two complements for each almost cluster tilting subcategory $\mathcal{T}'$, which we label $M$ and $M^{\ast}$ respectively.
We call the pair $(M,M^{\ast})$ an \textit{exchange pair}.

For every exchange pair $(M,M^{\ast})$ we have two non-split triangles, called the \textit{exchange triangles}:
\[
M \rightarrow B_{M^{\ast}} \rightarrow M^{\ast} \rightarrow \Sigma M \quad \text{and} \quad M^{\ast} \rightarrow B_{M} \rightarrow M \rightarrow \Sigma M^{\ast}
\]
where the maps $B_M \rightarrow M$ and $B_{M^{\ast}} \rightarrow M^{\ast}$ are right $\mathcal{T}$-approximations, where $\mathcal{T}$ is the cluster tilting subcategory containing both the almost cluster tilting subcategory $\mathcal{T}'$ and its complement $M$.

Notice that this a way of constructing a new cluster tilting subcategory of $\mathcal{C}_n$ from a given one, and is akin to a mutation of a seed in a cluster algebra.

Given some cluster tilting subcategory $\mathcal{T}$, with an exchange pair $(M, M^{\ast})$ and $M \in \mathcal{T}$, then by \cite{Gratz2017} the arc corresponding to $M^{\ast}$ must cross an arc corresponding to some indecomposable object $X \in \mathcal{T}$, as $\mathcal{T}$ corresponds to a maximal set of pairwise non-crossing arcs.
However, there exists a cluster tilting subcategory $\mathcal{S}$ with $M^{\ast} \in \mathcal{S}$ such that $\mathcal{S}$ and $\mathcal{T}$ share an almost cluster tilting subcategory $\mathcal{T}'$, which has complements $M$ and $M^{\ast}$.
This implies that the arc corresponding to $M^{\ast}$ must cross that corresponding to $M$.
Therefore, we know that $\operatorname{Ext}^1 (M,M^{\ast}) \cong k$, and therefore $\operatorname{dim}_k \operatorname{Hom}_{\mathcal{C}_n} (M, \Sigma M^{\ast})=1$.
By the symmetry of $\operatorname{Ext}^1(-,-)$ inherent in 2-Calabi-Yau triangulated categories, we also have $\operatorname{dim}_k \operatorname{Hom}_{\mathcal{C}_n} (M^{\ast}, \Sigma M)=1$.

It is important to note that by definition $B_M,B_{M^{\ast}} \in \mathcal{T}$, and the (potentially direct sum of) arcs they correspond to are exactly the arcs which form a quadrilateral with the arcs corresponding to $M$ and $M^{\ast}$.

\section{Grothendieck Group of \texorpdfstring{$\mathcal{C}_n$}{C(Z)}}\label{Sec:4}

In this section, we start by laying out the basic definitions of Grothendieck groups, alongside some results that will be of use to us in our computations of the Grothendieck group of $\mathcal{C}_n$.
Principally in this section, the main result which we use is a theorem due to Palu \cite{Palu2008}, which finds an isomorphism between the Grothendieck group of a Hom-finite, 2-Calabi-Yau triangulated category, with cluster tilting subcategories, and a quotient of the split Grothendieck group of a cluster tilting subcategory.

Then we move on to one of the main goals of this work, computing the Grothendieck group for each $\mathcal{C}_n$; unfortunately due to the nature of the proof in general, this will need to be considered in a slightly altered method for $n=1,2,3$, until it can be shown in generality for $n \geq 4$.
This is due to the general proof using a cluster tilting subcategory of $\mathcal{C}_n$ which corresponds to an inscribed $n$-gon of arcs in $S^1$, and using a triangulation of said $n$-gon to garner relations between the objects corresponding to the arcs on the edge of the $n$-gon.
As there is no triangulation of a $n$-gon for $n=1,2,3$, this necessitates a slightly altered proof in each case, although the general setup remains the same.

\subsection{The Grothendieck Group}

The Grothendieck group of a category is a way of constructing an abelian group on the set of objects in the category in a universal way.
This is a generalisation of the Grothendieck group over a commutative monoid, and obeys the same universal property as found in the commutative monoid setting.

The precise definition of a Grothendieck group depends on the structure that we are considering on our category; however all of these definitions follow the idea of taking the free group with basis the isoclasses of objects of the category, then taking the quotient by a set of relations $[Y]=[X]+[Z]$ for all triples $(X,Y,Z)$ satisfying a given condition.
For our purposes we only need give the definitions of a split and a triangulated Grothendieck group.

\begin{definition}
Let $\mathcal{T}$ be any triangulated category.
We denote by $G_0(\mathcal{T})$ the free abelian group with basis $\{[X]\; \vline \; X \in \mathcal{T}\}$, with $[X]$ corresponding to the isoclass of $X \in \mathcal{T}$.
Then the \textit{Grothendieck group of} $\mathcal{T}$, $K_0^{\operatorname{add}}(\mathcal{T})$, is defined to be the group
\[
K_0^{\operatorname{add}}(\mathcal{T})= G_0(\mathcal{T})/([Y]-[X]-[Z]\; \vline \; \text{such that} \; X \rightarrow Y \rightarrow Z \rightarrow \Sigma X \; \text{is a triangle in} \mathcal{T})
\]
\end{definition}
\begin{definition}
Let $\mathcal{S}$ be any additive category.
We denote by $G_0(\mathcal{S})$ the free abelian group with basis $\{[X]\}$, with $[X]$ corresponding to the isoclass of $X \in \mathcal{S}$.
Then we define the \textit{split Grothendieck group of} $\mathcal{S}$, $K_0^{\operatorname{add}}(\mathcal{S})$, to be
\[
K_0^{\operatorname{add}}(\mathcal{S})= G_0(\mathcal{S})/([X \oplus Y]-[X]-[Y]\; \vline \;X,\;Y \in \mathcal{S})
\]
\end{definition}

When working with the split Grothendieck group, we will want to use the following well-known result:

\begin{proposition} \label{Prop:SplitGroth} \cite[Proposition 2.1]{Naisse2019}   
Let $\mathcal{T}$ be a Krull-Schmidt category with a collection of non-isomorphic indecomposable objects $\{X_i\}_{i \in I}$.
The split Grothendieck group of $\mathcal{T}$ is the free abelian group 
\[
K_0^{add}(\mathcal{T}) \cong \bigoplus_{i \in I}\mathbb{Z} \cdot [X_i]
\]
with basis $\{[X_i]\}_{i \in I}$.
\end{proposition}

We are now able to state a theorem by Palu, which will be of use in the proof of Theorem \ref{Thm:TriGroGroupCn}.
\begin{theorem}\cite[Theorem 10]{Palu2008}\label{Thm:Palu}
Let $\mathcal{C}$ be a $\operatorname{Hom}$-finite, 2-Calabi-Yau triangulated category with a cluster tilting subcategory $\mathcal{T}$, and let $M \in \operatorname{ind}(\mathcal{T})$.
Let $B_M$ and $B_{M^{\ast}}$ be the central objects in the exchange triangles for the exchange pair $(M,M^{\ast})$.
Then, if $K_0^{\operatorname{add}}(\mathcal{T})$ is generated by all classes $[S_M]$ of simple $\mathcal{T}$-modules, the triangulated Grothendieck group of $\mathcal{C}$ is the quotient of the split Grothendieck group of the cluster tilting subcategory $\mathcal{T}$ by all relations $[B_{M^{\ast}}]- [B_M]$:
\[
K_0(\mathcal{C}) \cong K_0^{\operatorname{add}}(\mathcal{T}) / \langle ([B_{M^{\ast}}]- [B_M]) \; \vline \; M \in \operatorname{ind}(\mathcal{T}) \rangle  
\]
\end{theorem}
This allows us to compute the Grothendieck group of $\mathcal{C}_n$ as a quotient of the split Grothendieck group of a cluster tilting subcategory $\mathcal{T}$.

It should be noted that this is not the original statement in \cite{Palu2008}, as the condition that $K_0^{\operatorname{add}}(\mathcal{T})$ be generated by the simple modules in $\operatorname{mod} \mathcal{T}$ is implicitly used in the proof, but not explicitly stated as a requirement.
Here $\operatorname{mod} \mathcal{T}$ is used to denoted the category of $\mathcal{T}$-modules, i.e. the category of $k$-linear, contravariant functors from $\mathcal{T}$ to the category of $k$-vector spaces.
See Remark \ref{Rem:Palu} for further details.

\subsection{The Grothendieck Group of \texorpdfstring{$\mathcal{C}_n$}{Cn}}

As $\mathcal{C}_n$ is a Krull-Schmidt category, this implies by Proposition \ref{Prop:SplitGroth}  that $K_0^{\operatorname{add}}(\mathcal{T})$ is exactly the free abelian group generated by the indecomposable elements $X \in \mathcal{T}$.
This leaves us with only the exchange pairs of $\mathcal{T}$ to find, and thus we can compute the Grothendieck group of $\mathcal{C}_n$.
The exchange pairs can be easily computed with the following lemma.

\begin{lemma}\label{Lem:ExPairs}
Let $\mathcal{T}$ be a cluster tilting subcategory of $\mathcal{C}_n$, and let the corresponding collection of arcs be $\mathscr{T}$.
Let $X,Y,W,Z \in \mathcal{T}$ be indecomposable objects or zero objects, such that the corresponding arcs in $\mathscr{T}$ form a quadrilateral.
Let $M \in \operatorname{ind}(\mathcal{T})$ correspond to the arc in $\mathscr{T}$ the triangulates the quadrilateral, i.e.

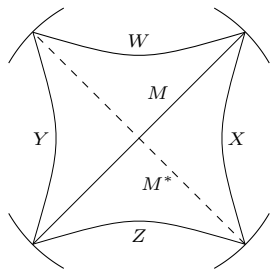
\[
\begin{tikzpicture}
   \draw [domain=30:60] plot ({2*cos(\x)}, {2*sin(\x)});
   \draw [domain=210:240] plot ({2*cos(\x)}, {2*sin(\x)});
    \draw [domain=300:330] plot ({2*cos(\x)}, {2*sin(\x)});
   \draw [domain=120:150] plot ({2*cos(\x)}, {2*sin(\x)});
   \draw[ultra thin] (1.41,1.41)..controls(1,0)..(1.41,-1.41);
   \draw[thin] (1.41,1.41)..controls(0,1)..(-1.41,1.41);
    \draw[ultra thin] (-1.41,-1.41)..controls(0,-1)..(1.41,-1.41);
   \draw[ultra thin] (-1.41,-1.41)..controls(-1,0)..(-1.41,1.41);
   \draw[ultra thin] (1.41,1.41) -- (-1.41,-1.41);
   \draw[thin, dashed] (1.41,-1.41) -- (-1.41,1.41);
   \node at (1.3,0) {\scriptsize $X$};
   \node at (0,1.3) {\scriptsize $W$};
    \node at (-1.3,0) {\scriptsize $Y$};
   \node at (0,-1.3) {\scriptsize $Z$};
   \node at (0.25,0.6) {\scriptsize $M$};
   \node at (0.25,-0.6) {\scriptsize $M^{\ast}$};
\end{tikzpicture}
\]
with $M^{\ast}$ such that $(M,M^{\ast})$ is an exchange pair.
Then
\[
B_M \cong W \oplus Z, \; B_{M^{\ast}} \cong X \oplus Y
\]
\end{lemma}

\begin{proof}
By definition, we have the two triangles 
\begin{align*}
    M \rightarrow X \oplus Y \rightarrow M^{\ast} \rightarrow \Sigma M\\
    M^{\ast} \rightarrow W \oplus Z \rightarrow M \rightarrow \Sigma M^{\ast}
\end{align*}
However we also have the exchange triangles 
\[
M \rightarrow B_{M^{\ast}} \rightarrow M^{\ast} \rightarrow \Sigma M \quad \text{and} \quad M^{\ast} \rightarrow B_{M} \rightarrow M \rightarrow \Sigma M^{\ast}
\]
By using the fact that $\operatorname{Ext}^1(M,M^{\ast}) \cong \operatorname{Ext}^1(M^{\ast},M) \cong k$, we now only need to show that none of these triangles split.

To split, the triangle
\[
M \xrightarrow{u} X \oplus Y \xrightarrow{v} M^{\ast} \xrightarrow{w} \Sigma M
\]
must have a retraction $v$, however we have
\[
\operatorname{Hom}(M^{\ast}, X \oplus Y) \cong \operatorname{Ext}^1(M^{\ast},\Sigma^{-1} X \oplus \Sigma^{-1} Y) =0
\]
where the final equality holds as neither $\Sigma^{-1} X$ nor $\Sigma^{-1} Y$ cross $M^{\ast}$, and so have trivial $\operatorname{Ext}$-spaces.
Hence there are no maps $v':M^{\ast} \rightarrow X\oplus Y$, and so $v$ cannot be a retract.
Thus the triangle doesn't split.

A similar argument works for the triangle
\[
M^{\ast} \rightarrow W \oplus Z \rightarrow M \rightarrow \Sigma M^{\ast}
\]

The two exchange triangles are non-split by definition.
Hence we have $B_M \cong W \oplus Z$ and $B_{M^{\ast}} \cong X \oplus Y$.
\end{proof}

\begin{theorem}\label{Thm:TriGroGroupCn}
Let $\mathscr{Z}$ be a collection of marked points on the circle with $n$ two-sided accumulation points.
Then the associated cluster category, $\mathcal{C}_n$ has the triangulated Grothendieck group:
\[
K_0(\mathcal{C}_n) \cong \mathbb{Z}^n
\]
\end{theorem}

\begin{proof}
One may see by Lemma \ref{Lem:ExPairs} that each leapfrog $\mathscr{L}_i$ induces two basis elements $[Z_i], [Y_i]$ in $K_0^{\operatorname{add}}(\mathcal{T})/\langle [B_M] - [B_{M^\ast}]\rangle$.
To see this, take some object $M \in \mathscr{L}_{i+1}$, with arcs $M^-,M^+ \in \mathscr{L}_{i+1}$ sharing an endpoint with $M$.
Then one of the pair of objects $B_M$ and $B_{M^\ast}$ coming from the exchange triangles must be trivial, and the other have two indecomposable direct summands, due to the construction of the leapfrog.
From this, and by the relation $[B_M]-[B_{M^{\ast}}]=0$, one may see that we get that we get $[M^-]=[M^+]$ using Lemma \ref{Lem:ExPairs}.
Then, via induction on the arcs in the leapfrog, one may see that every arc $M \in \mathscr{L}_{i+1}$ corresponds to either $[Z_i]$ or $[Y_i]$ in $K_0^{\operatorname{add}}(\mathcal{T})/\langle [B_M] - [B_{M^\ast}]\rangle$.

To prove the claim for all $n>0$, we shall prove it for $n=1,2,3$ respectively, and then for $n \geq 4$.

In the case $n=1$, we only have $a_1 \in L(\mathscr{Z})$.
Thus the first non-trivial arc in $L_1$ is the arc $\{z_1^-,z_1^+\}$, which we label $Z_1$ in lieu of the arc $\{z_1,z_2\}$.
Label the arc $\{z_1^{--},z_1^+\}$ by $Y_1$.
From this, one may see that the arc $Z_1^{\ast}$ has endpoints $\{z_1^{--},z_1\}$.
Using Lemma \ref{Lem:ExPairs} we see that $B_{Z_1} \cong Y_1$ and $B_{Z_1^{\ast}} = 0$, so we have $[B_{Z_1}] = [Y_1]$ and $[B_{Z_1^{\ast}}] = 0$, meaning $[B_{Z_1}]- [B_{Z_1^{\ast}}]=[Y_1]$. There are no more relations in $K_0^{\operatorname{add}}(\mathcal{T})$ left to consider, and so by Theorem \ref{Thm:Palu} \cite{Palu2008};
\[
K_0(\mathcal{C}_1) \cong \mathbb{Z}
\]

In the case of $n=2$, the arc $\{z_1,z_2\}$ is contained in both $L_1$ and $L_2$, and thus we have $Z_1 \cong Z_2$.
One can see from Lemma \ref{Lem:ExPairs} that $B_{Z_1} \cong Y_1 \oplus Y_2$ and $B_{Z_1^\ast} \cong 0$, and thus we have $[B_{Z_1}]-[B_{Z_1^{\ast}}]=[Y_1]+[Y_2]$, which means that when we quotient by $[B_{Z_1}]-[B_{Z_1^{\ast}}]$ we get $[Y_2]=-[Y_1]$.
Meaning we have only two basis elements in $K_0^{\operatorname{add}}(\mathcal{T})/\langle [B_M] - [B_{M^\ast}]\rangle$, and there are no more relations to consider.
Hence, by Theorem \ref{Thm:Palu} \cite{Palu2008};
\[
K_0 (\mathcal{C}_2) \cong \mathbb{Z}^2
\]

When we have $n=3$, we must consider the triangle consisting of sides being the arcs $\{z_i,z_{i+1}\}$, with $i=1,2,3$.
The exchange pair of $Z_1$ consists of $B_{Z_1} \cong Z_2 \oplus Y_1$ and $B_{Z_1^{\ast}} \cong Z_3$, giving us the relation $[Z_2]+[Y_1]=[Z_3]$.
Similarly, via the exchange pairs at $Z_2$ and $Z_3$ respectively, we have the relations,
\begin{align*}
    [Z_3] + [Y_2] = [Z_1]\\
    [Z_1] + [Y_3] = [Z_2]
\end{align*}
Which means that we have three basis elements of $K_0^{\operatorname{add}}(\mathcal{T})/\langle [B_M] - [B_{M^\ast}]\rangle$, which are the elements $\{[Z_1],[Z_2],[Z_3]\}$, with no more relations to consider, and so by Theorem \ref{Thm:Palu};
\[
K_0(\mathcal{C}_3) \cong \mathbb{Z}^3
\]

For $n \geq 4$, we have the additional arcs forming the fan triangulation of the inscribed $n$-gon to consider.
The corresponding objects to these arcs induce the elements $[X_i]$ in $K_0^{\operatorname{add}}(\mathcal{T})$, for all $i=3,\ldots,n-1$.
In these cases we wish to consider the exchange pairs on two different families of objects, $\{Z_i\}_{i=1,\ldots,n}$ and $\{X_j\}_{j = 3,\ldots,n-1}$.

For all $X_i$, $i=3, \ldots, n-1$ using Lemma \ref{Lem:ExPairs}, we have the exchange pairs $B_{X_i} \cong Z_i \oplus X_{i-1}$ and $B_{X_i^{\ast}} \cong X_{i+1} \oplus Z_{i-1}$.
These exchange pairs give us the relations;
\begin{align*}
    [Z_i] + [X_{i-1}] &= [Z_{i-1}] + [X_{i+1}], & \text{for all} \; i=3,\ldots,n-1
\end{align*}
where we use the fact that by definition $X_2 \cong Z_1$ and $X_n \cong Z_n$.
By inspection at the corresponding arc to $Z_i$ for $i=2,\ldots,n-1$, one can see that by Lemma \ref{Lem:ExPairs} we have the exchange pairs $B_{Z_i} \cong X_{i+1} \oplus Y_i$ and $B_{Z_i^{\ast}} \cong X_i$.
This means we have the set of relations
\begin{align*}
    [X_i]&=[X_{i+1}]+[Y_i]
\end{align*}
for $i=2,\ldots,n-1$.

It only remains to check the relations induced by the exchange pairs for $Z_1$ and $Z_n$.
One may easily check that we have the exchange pairs by applying Lemma \ref{Lem:ExPairs};
\begin{align*}
    B_{Z_1} &\cong Z_2 \oplus Y_1 &B_{Z_1^{\ast}} \cong X_3\\
    B_{Z_n} &\cong Y_n \oplus X_{n-1} &B_{Z_n^{\ast}} \cong Z_{n-1}
\end{align*}
These exchange pairs then induce the relations;
\begin{align*}
    [X_3]&=[Z_2]+[Y_1]\\
    [Z_{n-1}]&=[Y_n]+[X_{n-1}]
\end{align*}

This means that all of the relations $\langle B_M - B_{M^{\ast}} \rangle$, for $M$ an indecomposable summand of the cluster tilting object $T$, have been found.
These are
\begin{align}
    [Z_i] &= -[X_{i-1}] + [X_{i+1}] + [Z_{i-1}] & \text{for all} \; i=3,\ldots,n-1\\
    [Y_i]&=-[X_{i+1}]+[X_i] & \text{for all} \; i=2,\ldots,n-1\\
    [Z_2]&=[Y_1]+[X_3]\\
    [Y_n]&=[Z_{n-1}]-[X_{n-1}]
\end{align}
By considering  $(3)$ and then $(1)$ inductively, we find each $[Z_i]$ for $i=2,\ldots,n-1$ in terms of the elements $\{[Y_1], [X_2], \ldots, [X_n]\}$.
We may also find all $[Y_i]$ for $i=2,\ldots,n$ in terms of the same set of elements, by considering $(2)$ and $(4)$.
Note, given we also have $X_2 \cong Z_1$ and $X_n \cong Z_n$, so therefore $[X_2]=[Z_1]$ and $[X_n]=[Z_n]$, meaning we have found all elements in terms of linear combinations of the set of elements $\{[Y_1], [X_2], \ldots, [X_n]\}$.

Given there are no more relations to consider, this means we have a basis for the group $K_0^{\operatorname{add}}(\mathcal{T})/\langle B_M - B_{M^{\ast}} \rangle$;
\[
\{[Y_1],[X_2],\ldots, [X_n]\}
\]
All of these elements in the basis have infinite order, as there exists no relation such that $m[A] = 0$ for any $m \in \mathbb{Z}$ and $A \in \operatorname{ind}(\mathcal{T})$; and so, by Theorem \ref{Thm:Palu}
\[
K_0(\mathcal{C}_n) \cong \mathbb{Z}^n
\]

\end{proof}

\begin{remark}\label{Rem:Palu}
A natural question to ask is why this choice of cluster tilting subcategory is being used in particular.
We choose this cluster tilting subcategory primarily because it satisfies an implicit condition required to apply Theorem \ref{Thm:Palu}, which other cluster tilting subcategories do not necessarily satisfy.
This is that, given a cluster tilting subcategory $\mathcal{T}$, then $K_0^{\operatorname{add}}(\mathcal{T})$ must be generated by the simple $\mathcal{T}$-modules, $S_M$ associated to the indecomposable object $M$.

For example, if we were to take the additive subcategory, $\mathcal{S} \subset \mathcal{C}_1$, with indecomposable objects that correspond to a fountain at $z \in \mathscr{Z}$ converging to $a \in L(\mathscr{Z})$.
This is a cluster tilting subcategory by \cite{Gratz2017} when the corresponding arcs are maximal as a set of pairwise non-crossing arcs.
When we apply Theorem \ref{Thm:Palu} to $\mathcal{S}$ we get the result that $K_0(\mathcal{C}_1) \cong \mathbb{Z}^2$.
We know this cannot be true by considering only the triangles in $\mathcal{C}_1$ that correspond to the crossings of arcs, which shows that $K_0(\mathcal{C}_1)$ must be a quotient of $\mathbb{Z}$.

We also see that $\mathcal{S}$ does not satisfy the condition of $K_0^{\operatorname{add}}(\mathcal{S})$ is generated by simple $\mathcal{S}$-modules.
To do this, we can associate a quiver to a triangulation (see the survey article by Labardini-Fragoso \cite{Labardini13}), and for $\mathcal{S}$ this quiver has infinite length paths between vertices, and so a projective $\mathcal{S}$-module $P$ may not be the sum of finitely many simple $\mathcal{S}$-modules in $K_0^{\operatorname{add}}(\mathcal{S})$.

This is a condition that is mentioned within the proof of Theorem \ref{Thm:Palu} in \cite{Palu2008}, however is not mentioned within the statement of the theorem.
Whilst this does serve as a counterexample to the theorem as it is originally stated, the author still believes that a form of the theorem could still be used on cluster tilting subcategories without this condition, however the image of the map $\varphi$, as defined in \cite{Palu2008}, would have to be explicitly computed.

The author does not currently know of any 2-Calabi-Yau, triangulated category with at least one cluster tilting subcategory, such that there does not exist a cluster tilting subcategory that satisfies the conditions of Theorem \ref{Thm:Palu}.
\end{remark}

\section{A Completion of \texorpdfstring{$\mathcal{C}_n$}{Cn}}\label{Sec:5}

In this section we shall look at the completion of a discrete cluster category; which can be thought of as the discrete cluster category we started with, but with added extra "limit arcs".
These limit arcs can be thought of as the limit of a right fountain at some marked point $z$ to some accumulation point $a \in L(\mathscr{Z})$.
In the one accumulation point case, this is equivalent to formally adding certain homotopy colimits into the category.

\subsection{Completed Discrete Cluster Categories}

The completed discrete cluster categories of Dynkin type $A_{\infty}$ are a family of categories that are $\mathrm{Hom}$-finite, $k$-linear, Krull-Schmidt categories, corresponding to formally adding limit arcs to $\mathcal{C}(\mathscr{Z})$.
Crucially, they do not inherit the property of being 2-Calabi-Yau from $\mathcal{C}_n$, \cite{Paquette2020}, and so we cannot apply the theorem by Palu \cite{Palu2008} that we have used in the previous section.
We shall label the completed discrete cluster categories as $\overline{\mathcal{C}}_n$.

Both \cite{Paquette2020} and \cite{Fisher2014} introduce a completion of $\mathcal{C}_n$ via different approaches, however this currently only extends to a single accumulation point for Fisher's approach.
In \cite{Paquette2020}, Paquette and Yildirim construct $\overline{\mathcal{C}}_n$ by 'opening up' the accumulation points in the representation of $\mathcal{C}_n$ and adding arcs into them, before taking a localisation on a chosen set of morphisms $\Omega$.
This process is essentially the same as taking the corresponding inclusion functor of $\mathcal{C}_n$ into $\mathcal{C}_{2n}$, before taking the same localisation, which we will call $\pi : \mathcal{C}_{2n} \rightarrow \overline{\mathcal{C}}_n$.

Here we provide a more formal version of the construction.
For the full details and background, refer to \cite[Section 3]{Paquette2020}.
\begin{construction}\label{Cons:Cbar}
Let $S$ be the disc and we fix a collection of marked points $\mathscr{Z}$ on the boundary of $S$, such that there is a non-zero, finite set of accumulation points $L(\mathscr{Z})$.
Label the corresponding category $\mathcal{C}$.
We replace each $z_i \in L(\mathscr{Z})$ with an interval $[z_i^-,z_i^+]$ containing the points $z_{ij} \in (z_i^-,z_i^+)$ for $j \in \mathbb{Z}$ such that $z_{ij} < z_{ij'}$ if and only if $j < j'$.
We also set $\displaystyle \lim_{j \to +\infty} z_{ij} = z_i^+$ and $\displaystyle \lim_{j \to -\infty} z_{ij} = z_i^-$, as well as $\mathscr{Z}'=\mathscr{Z} \cup \{z_{ij}\; \vline \; j \in \mathbb{Z}, z_i \in L(\mathscr{Z})\}$.
Label the new category corresponding to this marked surface $\mathcal{C}'$.

Let $\mathcal{D}$ be the full additive subcategory of $\mathcal{C}'$ with indecomposable objects corresponding to arcs with both end points in the same interval $(z_i^-,z_i^+)$ for some $i$.
We choose a set of morphisms $\Omega$, by saying $f \in \Omega$ if there exists a triangle $X \xrightarrow{f} Y \rightarrow Z \rightarrow \Sigma X$ such that $Z \in \mathcal{D}$.
Then we define the completed category $\overline{\mathcal{C}}$ to be $\mathcal{C}' [\Omega^{-1}]$, the localisation of $\mathcal{C}'$ at
$\Omega$.
This comes with a localisation functor $\pi : \mathcal{C'} \rightarrow \overline{\mathcal{C}}$.
\end{construction}

Notice that in this construction we replace each accumulation point, $z_i \in L(\mathscr{Z})$, with two distinct accumulation points, $z_i^-$ and $z_i^+$.
This is where the intuition of the first part of the construction being equivalent to a fully faithful injection of objects in $n$ discrete copies of $\mathcal{C}_1$ into $\mathcal{C}_{2n}$ comes from.

A necessary step in the proof of Theorem \ref{Thm2} is to compute the Grothendieck group of the kernel of $\pi$.
To this, one can easily see $\mathrm{ker}(\pi)$ as a collection of copies of $\mathcal{C}_1$.

\begin{proposition}\label{Prop:ker}
Let $\pi: \mathcal{C}_{2n} \rightarrow \overline{\mathcal{C}}_n$ be the localisation in Construction \ref{Cons:Cbar}.
Then we have:
\[
\ker (\pi) \cong (\mathcal{C}_1)^n
\]
i.e. the kernel of $\pi$ is equivalent to $n$ copies of $\mathcal{C}_1$.
\end{proposition}

\begin{proof}
By the proof of \cite[Lemma 3.5]{Paquette2020}, we know that the kernel of $\pi$ is exactly $\mathcal{D}_n$, the thick subcategory of $\mathcal{C}_{2n}$ generated by the objects corresponding to the arcs with both of their endpoints between the same two accumulation points, on alternating segments, i.e.\ for each accumulation point only one side will correspond to objects in $\mathcal{D}_n$.
What remains now is to show that $\mathcal{D}_n= (\mathcal{C}_1)^n$, which we shall do inductively.

Firstly, let us take $n=1$, then we have $\mathcal{D}_1$ corresponding to the set of arcs contained between two accumulation points.
There is an obvious canonical functor that takes arcs in $\mathcal{C}_1$ to arcs in $\mathcal{D}_1$, which is one to one.
This functor is also clearly fully faithful, meaning that it is an equivalence of categories.

Next, we assume that the result holds for $n=k$, and show that it will hold for $n=k+1$.
We have the obvious injection of full subcategories,
\[
\mathcal{D}_k \rightarrow \mathcal{C}_{2k} \rightarrow \mathcal{C}_{2k+2}
\]
meaning that we have a fully faithful functor $\mathcal{D}_k \rightarrow \mathcal{D}_{k+1}$, and so it remains to show that $\mathcal{D}_{k+1}$ is equivalent to the sum of $\mathcal{D}_k$ and $\mathcal{C}_1$.
Clearly the functor from $\mathcal{D}_k$ to $\mathcal{D}_{k+1}$ has an image of $k$ of the $k+1$ segments in $\mathcal{D}_{k+1}$, meaning that we have one segment left, which is equivalent to $\mathcal{D}_1$, and hence $\mathcal{C}_1$ by the $n=1$ case.

The disjointedness comes the fact that no arcs on different segments cross, and therefore have trivial $\Ext$ spaces, and hence trivial $\Hom$ spaces too.
Thus we have $\mathcal{D}_{k+1}$ is the sum of $\mathcal{D}_k$ and $\mathcal{C}_1$, and is hence equivalent to $n$ copies of $\mathcal{C}_1$.
\end{proof}

\subsection{The Grothendieck Group of \texorpdfstring{$\overline{\mathcal{C}}_n$}{Cn}}

Here we compute the triangulated Grothendieck group of the completed discrete cluster category of Dynkin type $A_{\infty}$ with $n$ two-sided accumulation points, $\overline{\mathcal{C}}_n$.
A similar approach to the main result used in the previous section would not be feasible, as the results in \cite{Palu2008} require the category in question to be $2$-Calabi-Yau, a property that $\overline{\mathcal{C}}_n$ does not enjoy.
Therefore a new approach must be made, one for which the 2-Calabi-Yau property is not necessary.

This new approach involves using the localisation functor used in \cite{Paquette2020} to construct the category $\overline{\mathcal{C}}_n$ from the category $\mathcal{C}_{2n}$.
As shown in Proposition \ref{Prop:ker} this essentially surjective functor $\pi$ has a kernel of $n$ copies of the category $\mathcal{C}_1$, and so fits into the short exact sequence:
\[
0 \rightarrow (\mathcal{C}_1)^n \xlongrightarrow{\rho} \mathcal{C}_{2n} \xlongrightarrow{\pi} \overline{\mathcal{C}}_n \rightarrow 0
\]

Before we get to the main result and its proof, we must provide some useful statements that hold for both $\mathcal{C}_n$ and $\overline{\mathcal{C}}_n$.
We state these lemmas for $\mathcal{C}_n$, however note that the statements and proofs are exactly the same for $\overline{\mathcal{C}}_n$.

\begin{lemma}\label{Lem:markedp}
Let $W \in \mathcal{C}_n$ be any indecomposable object such that the arc corresponding to $W$ has both endpoints in the same segment.
Then the arc corresponding to $W$ has an odd number of marked points between its two endpoints if and only if $[W]\neq 0$.
\end{lemma}

\begin{proof}
We will show this inductively, simultaneously showing $[W]=0$ when the arc corresponding to $W$ has an even number of marked points between its endpoints.

Let $W_i$ correspond to an arc with $i \in \mathbb{Z}_{>0}$ marked points between its two endpoints, with all corresponding arcs to $W_i$ sharing one endpoint, i.e. $\{W_i\}_{i>0}$ is a one sided fountain.
It is clear to see that, as an arc, $W_0$ is isotopic to a boundary segment between two adjacent marked points, and therefore a zero object, so $[W_0]= 0$.

Next, we have $W_1$, to which we will assign $[W_1] \in K_0(\mathcal{C}_n)$.
We can then use the following triangle to find all subsequent $[W_i] \in K_0(\mathcal{C}_n)$;
\[
W_i \rightarrow W_{i+1} \rightarrow \Sigma^i W_1 \rightarrow \Sigma W_i
\]
This is easily verified as a distinguished triangle by construction of a quadrilateral of arcs, with the corresponding arcs of $W_i$ and $\Sigma^i W_1$ crossing.

From this, we can see that in $K_0(\mathcal{C}_n)$ we have;
\[
[W_{i+1}] = [W_i] + (-1)^{i}[W_1]
\]
and hence, using induction on $i$, we find 
\[ [W_i] =
  \begin{cases*}
    0 \quad& if $ i $ is even \\
    [W_1] & if $ i $ is odd
  \end{cases*}\]
\end{proof}

Recall from Figure \ref{fig:triangulation} and the proof of Theorem \ref{Thm:TriGroGroupCn} that we have the set of arcs $\{Y_1, X_2, \ldots, X_n\}$ that correspond to a basis in $K_0(\mathcal{C}_n)$.
For reference, we give a figure with these arcs on it.
\begin{figure}[h!]
\centering
\begin{tikzpicture}
\draw ({4*sin(350)},{4*cos(350)}) arc (100:440:4);
\draw[dotted] ({4*sin(10)},{4*cos(10)}) arc (80:100:4);
\draw[fill=white] ({4*sin(72)},{4*cos(72)}) circle (0.1cm);
\draw[fill=white] ({4*sin(144)},{4*cos(144)}) circle (0.1cm);
\draw[fill=white] ({4*sin(216)},{4*cos(216)}) circle (0.1cm);
\draw[fill=white] ({4*sin(288)},{4*cos(288)}) circle (0.1cm);
\draw ({4*sin(180)},{4*cos(180)}) .. controls(2,-2).. ({4*sin(98)},{4*cos(98)});
\draw ({4*sin(180)},{4*cos(180)}) -- ({4*sin(36)},{4*cos(36)});
\draw ({4*sin(180)},{4*cos(180)}) .. controls(-2,-2).. ({4*sin(252)},{4*cos(252)});
\draw ({4*sin(180)},{4*cos(180)}) -- ({4*sin(324)},{4*cos(324)});
\draw ({4*sin(180)},{4*cos(180)}) .. controls(2.2,-2.2).. ({4*sin(108)},{4*cos(108)});
\node at (2,-1.7) {$X_2$};
\node at (2.7,-2.3) {$Y_1$};
\node at (1,0) {$X_3$};
\node at (-0.8,0.1) {$X_{n-1}$};
\node at (-1.8,-2) {$X_n$};
\node at ({4.3*sin(180)},{4.3*cos(180)}) {$z_1$};
\node at ({4.3*sin(98)},{4.3*cos(98)}) {$z_2$};
\node at ({4.3*sin(36)},{4.3*cos(36)}) {$z_3$};
\node at ({4.3*sin(324)},{4.3*cos(324)}) {$z_{n-1}$};
\node at ({4.3*sin(252)},{4.3*cos(252)}) {$z_n$};
\end{tikzpicture}
\caption{The set of arcs corresponding to the chosen representatives of the basis of $K_0(\mathcal{C}_n)$.}
\label{fig:basis}
\end{figure}
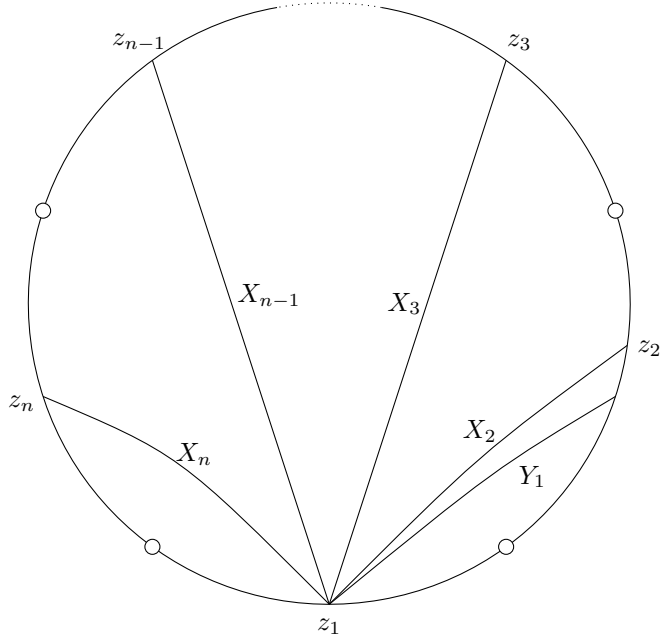

\begin{lemma}\label{Lemma:Triangle}
Let $W \in \mathcal{C}_{2n}$ correspond to any arc with both end points in the same segment.
Then $[W] \neq 0$ in $K_0(\mathcal{C}_{2n})$ if and only if there exists $j \in \mathbb{Z}$ such that $\Sigma^j W$ fits into a triangle of the following form
\[
X_i \rightarrow X \oplus \Sigma^j W \rightarrow \Sigma^l X_i\rightarrow \Sigma X_i
\]
for some $i\in \{2,\ldots,n\}$, with $l$ even and either $[X]=[X_2]+[Y_1]$ or $[\Sigma^j W]=[X_2]+[Y_1]$ when $W$ has endpoints on the same segment as $z_1$, where the objects $Y_1, X_2,\ldots,X_n$ correspond to the arcs in Figure \ref{fig:basis}.
\end{lemma}

\begin{proof}
($\Rightarrow$) Suppose we have $[W] \neq 0$, and we have $j \in \mathbb{Z}$ such that the corresponding arcs of $\Sigma^j W$ and $X_i$ share an endpoint.
Note that there will be two possible options for $j$, and the one we want depends upon out choice of representation of the basis for  $K_0(\mathcal{C}_n)$.
We one we want the one such that $\Sigma X_i$ will cross $\Sigma^j W$.

As $[W] \neq 0$, we know from Lemma \ref{Lem:markedp} that $W$ has an odd number of marked points between the endpoints of its corresponding arc, and hence shares an endpoint with the arc corresponding to $\Sigma^l X_i$, with $l$ being one more than the number of marked points between the endpoints of $W$, and hence even.
Thus, as $[X_i]$ and $[\Sigma^l X_i]$ cross, we have a triangle of the form
\[
X_i \rightarrow X \oplus \Sigma^j W \rightarrow \Sigma^l X_i\rightarrow \Sigma X_i
\]

All we have to show now is that $[X] = [X_2]+[Y_1]$.
To show this, consider that the arcs corresponding to $W$ and $X$ have the same number of marked  points between their endpoints, due to $X$ having both endpoints in the same segment (a different segment to the endpoints of $W$) and both $X$ and $W$ sharing an endpoint with both $X_i$ and $\Sigma^l X_i$.
This means that $[X] \neq 0$.
Now let us consider the triangle
\[
X_2 \rightarrow X \oplus U \rightarrow \Sigma^l Y_1 \rightarrow \Sigma X_2
\]
where $U$ is the unique object that makes this a distinguished triangle, up to isomorphism.

Given the suspension functor,$\Sigma$, acts on the arcs by rotating each endpoint to the next marked point in a clockwise direction, and the arc corresponding to $Y_1$ has one shared endpoint, $z_1$, with $X_2$ and the other endpoint of $Y_1$, $z_2^-$, is at the marked point directly clockwise of the other endpoint of $X_2$, $z_2$, then we have that $U$ has its endpoints on the same segment, and there is one $l$ marked points between its endpoints.
Hence, as $l$ is even, we have $[U]=0$ by Lemma \ref{Lem:markedp}, and so $[X]=[X_2]+[Y_1]$.

For the case where $[\Sigma^l W]=[X_2]+[Y_1]$, we simply assume the arc corresponding to $W$ is on the same segment as $z_1$, and so the same argument for $[X]=[X_2]+[Y_1]$ holds.

($\Leftarrow$) It is immediately obvious that we have $[W] + [X_2]+[Y_1] = [X_i] + (-1)^l[X_i]$, and so we have $[W] = 2[X_i] - [X_2]-[Y_1] \neq 0$, as the set $\{[X_i],[X_2],[Y_1]\}$ is linearly independent for $i=3,\ldots,n$, and if $i=2$ we have $[W]=[X_2]-[Y_1]\neq 0$. 
\end{proof}

With these results, we are now ready to state and prove the main result of this section.

\begin{theorem}\label{Thm2}
Let $\overline{\mathcal{C}}_n$ be the completion of a discrete cluster category of Dynkin type $A_{\infty}$ with $n$ two-sided accumulation points.
Then $\overline{\mathcal{C}}_n$ has the triangulated Grothendieck group:
\[
K_0(\overline{\mathcal{C}}_n) \cong \mathbb{Z}^n \oplus (\mathbb{Z}/2\mathbb{Z})^{n-1}
\]
\end{theorem}

\begin{proof}
We have the localisation functor $\pi \colon \mathcal{C}_{2n} \rightarrow \overline{\mathcal{C}}_n$, which has kernel $n$ copies of $\mathcal{C}_1$ by Proposition \ref{Prop:ker}.
This means that we have a short exact sequence of categories (in the language of \cite{Schlichting2003}):
\[
0 \rightarrow (\mathcal{C}_1)^n \xlongrightarrow{\rho} \mathcal{C}_{2n} \xlongrightarrow{\pi} \overline{\mathcal{C}}_n \rightarrow 0
\]
We also use the fact that the Grothendieck group functor, $K_0(-)$, is right exact \cite[Facts 1.2]{Schlichting2003}, and so that means we have the commutative diagram with exact rows:
\[
\begin{tikzcd}
& K_0((\mathcal{C}_1)^n) \arrow[r,"f"] \arrow[d, "\cong"] & K_0(\mathcal{C}_{2n}) \arrow[r,"g"] \arrow[d,"\cong"] & K_0(\overline{\mathcal{C}}_n) \arrow[r] \arrow[d,"\cong"] & 0\\
& \mathbb{Z}^n \arrow[r,"f"] & \mathbb{Z}^{2n} \arrow[r,"g"] & K_0(\overline{\mathcal{C}}_n) \arrow[r] & 0
\end{tikzcd}
\]
where the first and second isomorphisms come from Theorem \ref{Thm:TriGroGroupCn}.

Therefore, given $g$ is a surjection, it is sufficient to describe the map $f$ and find its cokernel.

We know that the functor $\rho: (\mathcal{C}_1)^n \rightarrow \mathcal{C}_{2n}$ is injective on objects, with each copy of $\mathcal{C}_1$ mapping into alternating segments in $\mathcal{C}_{2n}$.
This is due to the choice of $\mathcal{D}$ by Paquette and Yildirim in the construction of $\overline{\mathcal{C}}_n$ \cite{Paquette2020}.
We will label the basis elements of $K_0((\mathcal{C}_1)^n)$ as $\{ [M_1],\ldots, [M_n]\}$, with corresponding arcs $M_1,\ldots, M_n \in (\mathcal{C}_1)^n$.
Further, we shall choose the labelling such that $\rho (M_1) \cong X$, where $X$ is the object corresponding to the arc $\{z_1^{--},z_1\}$, and so we have $f([M_1])=[\rho(M_1)]=[X]$.
Generally we say that the arc $\{z_{2i-1}^{--},z_{2i-1}\}$ corresponds to the object $\rho(M_i)$, and so $[\rho(M_i)] \neq 0$ by Lemma \ref{Lem:markedp}.

Notice that if we choose an object, $W$, such that corresponds to the arc $\{z_i^{--},z_i\}$ and so $[W] \neq 0$ by Lemma \ref{Lem:markedp}, then $X$ is the unique object, corresponding to the arc $\{z_1^{--},z_1\}$, fitting into the distinguished triangle
\[
X_i \rightarrow X \oplus W \rightarrow \Sigma^2 X_i\rightarrow \Sigma X_i
\]
and so $[X]=[X_2]+[Y_1]$, by Lemma \ref{Lemma:Triangle}.

Given $[\rho(M_i)] \neq 0$ we may now apply Lemma \ref{Lemma:Triangle}, giving us the triangle
\[
X_{2i-1} \rightarrow X \oplus \rho(M_i) \rightarrow \Sigma^2 X_{2i-1} \rightarrow \Sigma X_{2i-1}
\]
with $[X]=[X_2]+[Y_1]$.
This in turn means we have

\[ [\rho(M_i)] =
  \begin{cases*}
    2[X_{2i-1}]-[X_2]-[Y_1] \quad& if $ i = 2,\ldots, n$ \\
    [X_2]+[Y_1] & if $ i =1$
  \end{cases*}\]
By summing $[\rho(M_i)]$ and $[\rho(M_1)]$, we get $[\rho(M_i)]+[\rho(M_1)]=2[X_{2i-1}]$, and so we get the image of $f$ being the span of the elements $\{[X_2] + [Y_1], 2[X_{2i-1}]\}$ for all $i=2,\ldots,n$.
This gives us $\operatorname{im}(f) \cong (2\mathbb{Z})^{n-1} \oplus \mathbb{Z}$, and so by the exact sequence
\[
\mathbb{Z}^n \xrightarrow{f} \mathbb{Z}^{2n} \xrightarrow{g} K_0(\overline{\mathcal{C}_n}) \rightarrow 0
\]
we get $K_0(\overline{\mathcal{C}_n}) \cong \mathbb{Z}^{2n}/ \operatorname{ker}(g) \cong \mathbb{Z}^{2n}/ \operatorname{im}(f)$, where the last equivalence comes from the above sequence being exact.
Hence we get

\[
K_0(\overline{\mathcal{C}}_n) \cong \mathbb{Z}^n \oplus (\mathbb{Z}/2\mathbb{Z})^{n-1}
\]
\end{proof}

\printbibliography

\end{document}